\documentclass[12pt]{article}
\usepackage{amssymb}
\usepackage{amsfonts}
\usepackage{amsmath}

\newtheorem{theorem}{Theorem}
\newtheorem{acknowledgement}[theorem]{Acknowledgement}

\newtheorem{conjecture}[theorem]{Conjecture}
\newtheorem{corollary}[theorem]{Corollary}

\newtheorem{lemma}[theorem]{Lemma}

\newenvironment{proof}[1][Proof]{\noindent\textbf{#1.} }{\ \rule{0.5em}{0.5em}}
\input{tcilatex}

\begin{document}

\title{Additive colorings of planar graphs}
\author{Jaros\l aw Grytczuk \\
{\small Faculty of Mathematics and Information Science,}\\
{\small \ Warsaw University of Technology, 00-661 Warszawa, Poland}\\
{\small \ e-mail: grytczuk@mini.pw.edu.pl} \and Tomasz Bartnicki, Sebastian
Czerwi\'{n}ski \\
{\small Faculty of Mathematics, Computer Science, and Econometrics,}\\
{\small \ University of Zielona G\'{o}ra, 65-516 Zielona G\'{o}ra, Poland}\\
{\small \ e-mail: t.bartnicki@wmie.uz.zgora.pl, s.czerwinski@wmie.uz.zgora.pl%
} \and Bart\l omiej Bosek, Grzegorz Matecki, Wiktor \.{Z}elazny \\
{\small Faculty of Mathematics and Computer Science,}\\
{\small \ Jagiellonian University, 30-348 Krak\'{o}w, Poland}\\
{\small \ e-mail: bosek@tcs.uj.edu.pl, matecki@tcs.uj.edu.pl,
zelazny@tcs.uj.edu.pl}}
\maketitle

\begin{abstract}
An \emph{additive coloring} of a graph $G$ is an assignment of positive
integers $\{1,2,\ldots ,k\}$ to the vertices of $G$ such that for every two
adjacent vertices the sums of numbers assigned to their neighbors are
different. The minimum number $k$ for which there exists an additive
coloring of $G$ is denoted by $\eta (G)$. We prove that $\eta (G)\leqslant
468$ for every planar graph $G$. This improves a previous bound $\eta
(G)\leqslant 5544$ due to Norin. The proof uses Combinatorial
Nullstellensatz and coloring number of planar hypergrahs. We also
demonstrate that $\eta (G)\leqslant 36$ for $3$-colorable planar graphs, and 
$\eta (G)\leqslant 4$ for every planar graph of girth at least $13$. In a
group theoretic version of the problem we show that for each $r\geqslant 2$
there is an $r$-chromatic graph $G_{r}$ with no additive coloring by
elements of any Abelian group of order $r$.
\end{abstract}

\section{Introduction}

Let $G$ be a simple graph, and let $k$ be a positive integer. By a \emph{%
coloring} of $G$ we mean any function $f$ from the set of vertices $V(G)$ to
the set $\{1,2,\ldots ,k\}$. Given a coloring $f$, consider the induced
function $S=S(f)$ on the set $V(G)$ defined by the formula%
\begin{equation*}
S(v)=\sum_{x\in N(v)}f(x),
\end{equation*}%
where $N(v)$ denotes the set of neighbors of the vertex $v$ in $G$. The
initial coloring $f$ is called an \emph{additive coloring} of $G$ if $%
S(u)\neq S(v)$ for every pair of adjacent vertices $u$ and $v$. The minimum
number $k$ for which there exists an additive coloring of $G$ is denoted by $%
\eta (G)$.

The notion of additive coloring was introduced in \cite{CzerwinskiGZ} as a
vertex version of the 1-2-3-conjecture of Karo\'{n}ski, \L uczak, and
Thomason \cite{KaronskiLT}. In the original problem the numbers are assigned
to the edges of a graph, and prospective color of a vertex $v$ is derived as
the sum of numbers assigned to the edges incident to $v$. It is conjectured
that for every connected graph (except $K_{2}$) one can produce a proper
vertex coloring in this way using only three numbers--- 1, 2, and 3.
Currently best bound is $5$, as proved by Kalkowski, Karo\'{n}ski, and
Pfender \cite{KalkowskiKP}.

In the related additive coloring problem no finite bound is possible since
for cliques we have $\eta (K_{n})=n$. We conjecture however, that perhaps $%
\eta (G)\leqslant \chi (G)$ for every graph $G$, where $\chi (G)$ denotes
the usual chromatic number. This conjecture is widely open as it is not
known whether $\eta (G)$ is bounded for bipartite graphs. In \cite%
{CzerwinskiGZ} we proved that $\eta (G)\leqslant 3$ for planar bipartite
graphs, and also that $\eta (G)\leqslant 100280245065$ for general planar
graphs. The later bound was improved to $5544$ by Norin (personal
communication). We present this proof in section 2 for completeness.

In this note we obtain a further improvement of this bound. Our main result
asserts that $\eta (G)\leqslant 468$ for every planar graph $G$. The proof
uses Combinatorial Nullstellensatz of Alon \cite{AlonCN}, and the coloring
number of hyperhraphs represented by planar bipartite graphs. For planar
graphs of girth at least $13$ we get a much better bound by $4$, using a
decomposition theorem from \cite{BuCranston}.

\section{Coloring number of graphs and hypergraphs}

We start with presenting an unpublished result of Norin. Recall that the
coloring number $\func{col}(G)$ of a graph $G$ is the least integer $k$ such
that there exists a linear ordering of the vertices $v_{1},\ldots ,v_{n}$
such that the number of \emph{backward} neighbors of $v_{i}$ (those
contained in the set $\{v_{1},\ldots ,v_{i-1}\}$) is at most $k-1$, for
every $i=1,2,\ldots ,n$. It is well known that $\func{col}(G)\leqslant 6$
for every planar graph $G$.

\begin{theorem}
\label{Norin}\emph{(S. Norin)} Let $G$ be a graph with chromatic number $%
\chi (G)=r$ and coloring number $\func{col}(G)=k$. Let $n_{1},\ldots ,n_{r}$
be $r$ pairwise coprime integers, with $n_{i}\geqslant k$ for all $%
i=1,2,\ldots ,k$. Then $\eta (G)\leqslant n_{1}\times \ldots \times n_{r}$.
In particular, $\eta (G)\leqslant 5544$ for every planar graph $G$ (by
taking $n_{1}=7$, $n_{2}=8$, $n_{3}=9$, and $n_{4}=11$).
\end{theorem}

\begin{proof}
Fix a proper coloring $c$ of a graph $G$ using colors $\{1,2,\ldots ,r\}$.
Also, fix a linear ordering of the vertices realizing $\func{col}(G)=k$. Let 
$n_{1},\ldots ,n_{r}$ be any positive integers such that $\gcd
(n_{i},n_{j})=1$ for every pair $i\neq j$, with $n_{i}\geqslant k$ for all $%
i=1,2,\ldots ,r$. Suppose now that each vertex $v$ is assigned with certain
weight $n(v)\in \mathbb{Z}_{n_{j}}$, with $j=c(v)$. Denote by $S_{i}(v)$ the
sum of weights of all the neighbors of $v$ in color $i$. More formally,%
\begin{equation*}
S_{i}(v)=\dsum\limits_{x\in N(v)\cap c^{-1}(i)}n(x),
\end{equation*}%
where summation is in the group $\mathbb{Z}_{n_{i}}$. Finally, let $%
S(v)=(S_{1}(v),\ldots ,S_{r}(v))$.

Since no neighbor of $v$ is colored with $c(v)$, we have $S_{j}(v)=0$ for $%
j=c(v)$. Our aim is to modify weights $n(v)$ greedily so that $%
S_{c(v)}(u)\neq 0$ for every backward neighbor $u$ of $v$. This will imply
that $S(u)\neq S(v)$ for every pair of adjacent vertices $u$ and $v$.

Suppose we have achieved this property for all vertices up to $v_{i-1}$ by
choosing appropriate weights $n(v_{1}),\ldots ,n(v_{i-1})$. Now we have to
find a weight for the vertex $v_{i}$. Let $j=c(v_{i})$. For every backward
neighbor $u$ of $v_{i}$ there is only one value of $n(v_{i})$ making $%
S_{j}(u)=0(\func{mod}n_{j})$. Since there are at most $k-1$ backward
neighbors of $v_{i}$, there are only $k-1$ forbidden values for $n(v_{i})$.
Since $n_{j}>k-1$, there is a free element of $\mathbb{Z}_{n_{j}}$ for the
weight $n(v_{i})$.

To get an additive coloring of graph $G$ we assign to every vertex $v$, an
element $f(v)=(f_{1}(v),\ldots ,f_{r}(v))$ of the group $\mathbb{Z}%
_{n_{1}}\times \ldots \times \mathbb{Z}_{n_{r}}$, defined by $f_{j}(v)=n(v)$
if $j=c(v)$, and $f_{j}(v)=0$, otherwise. This completes the proof, as the
group $\mathbb{Z}_{n_{1}}\times \ldots \times \mathbb{Z}_{n_{r}}$ is
isomorphic to $\mathbb{Z}_{N}$, where $N=n_{1}\times \ldots \times n_{r}$.
\end{proof}

The notion of coloring number can be generalized in a natural way for
hypergraphs. Given a hypergraph $H$ and a linear ordering of the vertices $%
v_{1},\ldots ,v_{n}$, define the \emph{backward degree} of vertex $v_{i}$ as
the number of \emph{different} hyperedges of the form $\{v_{j}\}\cup A$,
with $A\subseteq \{v_{1},\ldots ,v_{i-1}\}$ (we allow $A$ to be empty). The 
\emph{coloring number} $\func{col}(H)$ of hypergraph $H$ is the minimum $k$
such that in some linear ordering of the vertices all backward degrees are
at most $k-1$. This definition differs slightly from the one given in \cite%
{Kierstead}, but it is appropriate for our purposes.

\begin{lemma}
\label{col(H)}Let $H$ be a hypergraph with $\func{col}(H)=k$. Then there is
a function $f:V(H)\rightarrow \mathbb{Z}_{k}$ such that every hyperedge $B$
satisfies%
\begin{equation*}
\dsum\limits_{v\in B}f(v)\neq 0(\func{mod}k).
\end{equation*}
\end{lemma}

\begin{proof}
Start with a linear ordering of the vertices realizing $\func{col}(H)$ and
proceed greedily in that order. At each step there are at most $k-1$ partial
sums we have to account, and each of them is reset by exactly one value.
Hence, there is always a good choice for the next value of $f$.
\end{proof}

Now we give an upper bound for the coloring number of hypergraphs arising
from bipartite planar graphs.

\begin{lemma}
\label{colBP}Let $G$ be a bipartite planar graph with bipartition classes $X$
and $Y$. Let $H$ be a hypergraph on the set of vertices $X$ whose incidence
graph is $G$. Then $\func{col}(H)\leqslant 12$. In particular, there exists
a coloring $f:X\rightarrow $ $\mathbb{Z}_{12}$ satisfying condition:%
\begin{equation*}
\dsum\limits_{x\in N(y)}f(x)\neq 0(\func{mod}12)
\end{equation*}%
for every non-isolated vertex $y\in Y$.
\end{lemma}

\begin{proof}
We may assume that no two vertices in $Y$ are \emph{twins} (have exactly the
same nonempty neighborhood), as multiple hyperedges do not count in backward
degree. We shall prove that hypergraph $H$ always contains a vertex of the
usual degree at most $11$. This is sufficient since a hypergraph $H-x$ still
does not contain multiple hyperedges, (therefore the incidence graph of $H-x$
does not contain twins) and we may order the vertices of $H$ by sequential
deletion of such vertices.

Fix an embedding of $G$ in the plane. Transform this embedding into a new
plane graph $P$ in the following way. For every vertex $y\in Y$, draw a
simple closed curve $C(y)$ through the neighbors of $y$ within $\varepsilon $%
-distance from the connecting edges, so that a simply connected region $F(y)$
arises with the following properties:

\begin{enumerate}
\item All neighbors of $y$ belong to $C(y)$.

\item All other points of the edges connecting $y$ to its neighbors (and $y$
itself) are in the interior of $F(y)$.

\item No other points of the embedding of $G$ are in $F(y)$.
\end{enumerate}

Forget now about $y$'s and their edges inside regions $F(y)$. In this way we
get a plane (pseudo)graph $P$ on the set of vertices $X$ whose faces can be
properly $2$-colored: color the faces $F(y)$ by black and all other faces by
white. Notice that hypergedes of $H$ turned into black faces in $P$. Hence, $%
\deg _{H}(v)$ is just the number of black faces incident to $v$.

We claim that there is always a vertex in $P$ incident to at most $11$ black
faces. First, shrink all loops and all $2$-sided faces of $P$ to get a new
pseudograph $Q$ whose faces have at least three vertices. Let $v$, $e$, and $%
f$ denote the number of vertices, edges, and faces in $Q$, respectively. So,
we have $3f\leqslant 2e$, and by Euler's formula we get $e\leqslant 3v-6$.
Hence, there must be a vertex $x$ of degree at most $5$ in $Q$. Now, by the
lack of twins in $G$, each edge incident to $x$ in $Q$ has multiplicity at
most $4$ in $P$. Also, there can be at most one loop at each vertex in $P$,
by the same reason. Therefore degree of $x$ in $P$ is at most $22$, and
there are at most $11$ black faces incident to $x$. The proof of the lemma
is complete.
\end{proof}

It is worth noticing that the above lemma is tight. To see this take the
icosahedron on the vertex set $X$ and modify it in the following way: (1)
subdivide each edge and each face of the icosahedron with one new vertex,
(2) append a hanging edge to each vertex from $X$. The resulting graph is a
twin-free planar bipartite graph in which every vertex in $X$ has degree $11$%
.

\section{Combinatorial Nullstellensatz}

For the proof of our main result we will need a simple consequence of the
celebrated Combinatorial Nullstellensatz of Alon. For the sake of
completeness we provide also an elegant, simple proof due to Micha\l ek \cite%
{Michalek}.

\begin{theorem}
\label{Null}\emph{(Combinatorial Nullstellensatz)} Let $\mathbb{F}$ be an
arbitrary field, and let $P(x_{1},\ldots ,x_{n})$ be a polynomial in the
ring of polynomials $\mathbb{F}[x_{1},\ldots ,x_{n}]$. Suppose that there is
a nonvanishing monomial $x_{1}^{k_{1}}\ldots x_{n}^{k_{n}}$ in $P$ such that 
$k_{1}+\ldots +k_{n}=\deg (P)$. Then for every subsets $A_{1},\ldots ,A_{n}$
of the field $\mathbb{F}$, with $\left\vert A_{i}\right\vert \geqslant
k_{i}+1$, there are elements $a_{i}\in A_{i}$ such that $P(a_{1},\ldots
,a_{n})\neq 0$.
\end{theorem}

\begin{proof}
We will proceed by induction on the degree of polynomial $P$. If $\deg (P)=0$%
, then $P$ is a nonzero constant polynomial and the assertion holds
trivially. Let $\deg (P)\geqslant 1$ and suppose the theorem is true for all
polynomials of strictly smaller degree. Hence, for at least one $i\in
\{1,\ldots ,n\}$ we must have $k_{i}\geqslant 1$. Assume, for simplicity,
that $k_{1}\geqslant 1$, and let $a\in A_{1}$ be a fixed element. Using the
algorithm of long division of polynomials, we may write%
\begin{equation}
P=(x_{1}-a)Q+R.  \tag{*}
\end{equation}

Indeed, we may treat $P$ as a polynomial in one variable $x_{1}$ with
coefficients in the ring $\mathbb{F}[x_{2},\ldots ,x_{n}]$ and perform long
division by the polynomial $(x_{1}-a)$ to determine uniquely quotient $Q$
and remainder $R$. Since $\deg (x_{1}-a)=1$, the remainder $R$ must be a
constant in $\mathbb{F}[x_{2},\ldots ,x_{n}]$, which means that it does not
contain variable $x_{1}$. Hence, by the assumption on the nonvanishing
monomial in $P$, the quotient $Q$ must have a nonvanishing monomial $%
x^{k_{1}-1}x_{2}^{k_{2}}\ldots x_{n}^{k_{n}}$ and $\deg
(Q)=(k_{1}-1)+k_{2}+\ldots +k_{n}$.

Suppose on the contrary that $P(x)$ vanishes on the set $A_{1}\times \ldots
\times A_{n}$. Take any element $x\in \{a\}\times A_{2}\times \ldots \times
A_{n}$ and substitute to equation $(\ast )$. Since $P(x)=0$, we get that $%
R(x)=0$. But $R$ does not contain variable $x_{1}$, so it follows that $R$
also vanishes on the whole set $A_{1}\times \ldots \times A_{n}$. Take now
any $x\in (A_{1}\setminus \{a\})\times A_{2}\times \cdots \times A_{n}$ and
substitute to equation $(\ast )$. Since $P(x)=0$, $R(x)=0$, and $%
(x_{1}-a)\neq 0$, it follows that $Q(x)=0$. This means that $Q$ vanishes on
the whole set $(A_{1}\setminus \{a\})\times A_{2}\times \cdots \times A_{n}$%
, which contradicts the inductive assumption.
\end{proof}

The above theorem has many surprising applications in geometry,
combinatorics, and number theory \cite{AlonCN}. We used it in \cite%
{CzerwinskiGZ} to prove that every planar bipartite graph has an additive
coloring from arbitrary lists of size at least three. Below we give a slight
extension of this result, which will be useful later.

\begin{theorem}
\label{ListBip}Let $G$ be a bipartite graph whose edges can be oriented so
that each vertex has indegree at most $k$. Suppose that each vertex $v$ is
assigned with a list \thinspace $L(v)$ of $k+1$ real numbers. Then for every
function $q:V(G)\rightarrow \mathbb{R}$ there is a coloring $f$ of the
vertices such that%
\begin{equation*}
q(u)+\dsum\limits_{x\in N(u)}f(x)\neq q(v)+\dsum\limits_{x\in N(v)}f(x)
\end{equation*}%
for every pair of adjacent vertices $u$ and $v$.
\end{theorem}

\begin{proof}
Let $U=\{u_{1},\ldots ,u_{m}\}$ and $V=\{v_{1},\ldots ,v_{n}\}$ be the
bipartition classes of a graph $G$. Let $\{x_{1},\ldots ,x_{m}\}$ and $%
\{y_{1},\ldots ,y_{n}\}$ be the variables assigned to the vertices of these
classes, respectively. Denote by $S(u)$ the sum of variables assigned to the
neighbors of $u$. Consider a polynomial $P$ over the field of reals defined
by%
\begin{equation*}
P(x_{1},\ldots ,x_{m},y_{1},\ldots ,y_{n})=\dprod\limits_{u_{i}v_{j}\in
E(G)}(q(u_{i})+S(u_{i})-q(v_{j})-S(v_{j})).
\end{equation*}%
We claim that $P$ contains a nonvanishing monomial with exponents bounded by 
$k$. Let $\overrightarrow{G}$ be an orientation of $G$ with indegrees
bounded by $k$. In every factor of $P$ corresponding to edge $u_{i}v_{j}$
choose one of the variables $x_{i}$ or $y_{j}$--the one that corresponds to
the vertex on which the arrow points. In this way we obtain monomial $%
M=x_{1}^{k_{1}}\ldots x_{m}^{k_{m}}y_{1}^{l_{1}}\ldots y_{n}^{l_{n}}$
satisfying $0\leqslant k_{i},l_{j}\leqslant k$. Why is this monomial
nonvanishing in $P$? It is because each variable occurs in factors of $P$
with uniform sign ($x_{i}$ with minus sign, $y_{j}$ with plus sign). Hence,
the sign of monomial $M$ in $P$ is uniquely determined by the sequence of
exponents, and therefore its copies cannot cancel. Finally, to apply
Combinatorial Nullstellensatz, notice that $\deg (P)$ equals the number of
edges in $G$, which is the same as $k_{1}+\ldots +k_{m}+l_{1}+\ldots +l_{n}$
since $q(u_{i})-q(v_{j})$ are constants.
\end{proof}

\begin{corollary}
\label{TreeBP}Every tree has an additive coloring from arbitrary lists of
size two. Every bipartite planar graph has an additive coloring from
arbitrary lists of size three.
\end{corollary}

\begin{proof}
Every tree has an orientation with at most one incoming edge to every
vertex. Every bipartite planar graph has an orientation with indegrees
bounded by two.
\end{proof}

\section{Main results}

Let us start with a simpler case of planar $3$-colorable graphs.

\begin{theorem}
Every planar graph $G$ with $\chi (G)\leqslant 3$ satisfies $\eta
(G)\leqslant 36$.
\end{theorem}

\begin{proof}
Let $V(G)=A\cup B\cup C$ be a partition of the vertex set of $G$ into three
independent sets. Let $H$ be a subgraph of $G$ on the set of vertices $%
V(H)=A\cup B\cup C$ with the edge set%
\begin{equation*}
E(H)=\{uv\in E(G):u\in A\cup B\text{ and }v\in C\}.
\end{equation*}%
Clearly $H$ is a bipartite graph. Hence, by Theorem \ref{colBP}, there is a
function $h:C\rightarrow \{1,2,\ldots ,12\}$ such that the sum 
\begin{equation*}
S_{h}(u)=\dsum\limits_{x\in N_{H}(u)}h(x)
\end{equation*}%
satisfies $S_{h}(u)\neq 0(\func{mod}12)$ for every vertex $u\in A\cup B$
having at least one neighbor in $C$. For other vertices the above sum is
empty and we adopt $S_{h}(u)=0$ by convention.

Consider now a bipartite subgraph $F$ of $G$ induced by the vertices $A\cup
B $. Assign to each vertex $u$ in $F$ the list $L(u)=\{12,24,36\}$, and
apply Theorem \ref{ListBip} with function $q(u)=S_{h}(u)$. Let $f$ be a
coloring satisfying the assertion of Theorem \ref{ListBip}. That is, $f$
satisfies condition $S_{f}(u)+S_{h}(u)\neq S_{f}(v)+S_{h}(v)$ for every edge 
$uv\in E(F)$, where%
\begin{equation*}
S_{f}(u)=\dsum\limits_{x\in N_{F}(u)}f(x).
\end{equation*}%
Finally, let $g$ be a function defined on the whole set of vertices $V(G)$
by joining $f$ and $h$:%
\begin{equation*}
g(x)=\left\{ 
\begin{array}{l}
h(x)\text{ if }x\in C \\ 
f(x)\text{ if }x\in A\cup B%
\end{array}%
\right. .
\end{equation*}%
We claim that $g$ is an additive coloring of $G$ over the set $\{1,2,\ldots
,36\}$. Let $S(u)$ be the sum of $g$-labels over the whole neighborhood $%
N(u) $, that is, $S(u)=S_{h}(u)+S_{f}(u)$. Let $uv$ be any edge in $G$. If $%
u\in A\cup B$ and $v\in C$, then $S_{h}(u)\neq 0(\func{mod}12)$ and $%
S_{f}(u)=0(\func{mod}12)$, thus $S(u)\neq 0(\func{mod}12)$. On the other
hand, $S_{h}(v)=S_{f}(v)=0(\func{mod}12)$, so $S(v)=0(\func{mod}12)$. In the
other case, if $u\in A$ and $v\in B$, condition $S(u)\neq S(v)$ is
guaranteed by construction of $f$. This completes the proof.
\end{proof}

The proof for $4$-colorable planar graphs is similar in spirit, though a bit
more technical.

\begin{theorem}
Every planar graph satisfies $\eta (G)\leqslant 468$.
\end{theorem}

\begin{proof}
Let $V(G)=A\cup B\cup C\cup D$ be a partition of the vertex set of $G$ into
four independent sets. Let $H_{1}$ be a subgraph of $G$ on the set of
vertices $(A\cup B)\cup C$ with the edge set%
\begin{equation*}
E(H_{1})=\{uv\in E(G):u\in A\cup B\text{ and }v\in C\}.
\end{equation*}%
Clearly $H_{1}$ is a bipartite graph. Hence, by Theorem \ref{colBP}, there
is a function $h_{1}:C\rightarrow \mathbb{Z}_{12}$ such that the sum 
\begin{equation*}
S_{h_{1}}(u)=\dsum\limits_{x\in N_{H_{1}}(u)}h_{1}(x)
\end{equation*}%
satisfies $S_{h_{1}}(u)\neq 0(\func{mod}12)$ for every vertex $u\in A\cup B$
with at least one neighbor in $C$. Now, Let $H_{2}$ be a subgraph of $G$ on
the set of vertices $(A\cup B\cup C)\cup D$ with the edge set%
\begin{equation*}
E(H_{2})=\{uv\in E(G):u\in A\cup B\cup C\text{ and }v\in D.
\end{equation*}%
Clearly $H_{2}$ is a bipartite graph. Hence, by Theorem \ref{colBP}, there
is a function $h_{2}:D\rightarrow $ $\mathbb{Z}_{13}$ such that the sum 
\begin{equation*}
S_{h_{2}}(u)=\dsum\limits_{x\in N_{H_{2}}(u)}h_{2}(x)
\end{equation*}%
satisfies $S_{h_{2}}(u)\neq 0(\func{mod}13)$ for every vertex $u\in (A\cup
B\cup C)$ having a neighbor in $D$.

Now, using functions $h_{1}$ and $h_{2}$, we define a new function $h:C\cup
D\rightarrow \{1,2,\ldots ,156\}$ as follows. First we extend $h_{1}$ and $%
h_{2}$ to the whole set $C\cup D$ by putting $h_{1}(x)=0$ for $x\in D$ and $%
h_{2}(x)=0$ for $x\in C$. Let $\sigma $ be a group isomorphism from $\mathbb{%
Z}_{12}\times \mathbb{Z}_{13}$ to $\mathbb{Z}_{156}$. For each $x\in C\cup D$
define $h(x)$ as the unique number in the range $\{1,2,\ldots ,156\}$
satisfying congruence%
\begin{equation*}
h(x)\equiv \sigma ((h_{1}(x),h_{2}(x))(\func{mod}156).
\end{equation*}%
Let%
\begin{equation*}
S_{h}(u)=\dsum\limits_{x\in N(u)\cap (C\cup D)}h(x)
\end{equation*}%
for every $u\in A\cup B$, where, as before, $S_{h}(u)=0$ if $N(u)\cap (C\cup
D)=\emptyset $. First we claim that $S_{h}(u)\neq 0(\func{mod}156)$ for
every vertex $u\in A\cup B$ which has at least one neighbor in $C\cup D$.
Indeed, since $\sigma $ is a group isomorphism we may write%
\begin{eqnarray*}
S_{h}(u) &=&\dsum\limits_{x\in N(u)\cap (C\cup D)}h(x)=\dsum\limits_{x\in
N(u)\cap (C\cup D)}\sigma ((h_{1}(x),h_{2}(x)) \\
&=&\sigma \left( \left( \dsum\limits_{x\in N(u)\cap
C}h_{1}(x),\dsum\limits_{x\in N(u)\cap D}h_{2}(x)\right) \right) =\sigma
((S_{h_{1}}(u),S_{h_{2}}(u))).
\end{eqnarray*}%
Hence, $S_{h}(u)$ cannot be zero in $\mathbb{Z}_{156}$, since at least one
of the sums $S_{h_{1}}(u)$ or $S_{h_{2}}(u)$ is non-zero in its respective
group. Notice also that $S_{h}(u)\neq 0(\func{mod}156)$ for every vertex $%
u\in C$ and having a neighbor in $D$, as in this case we have $%
S_{h}(u)=\sigma ((0,S_{h_{2}}(u)))$ and $S_{h_{2}}(u)\neq 0$ in $\mathbb{Z}%
_{13}$.

Consider now a bipartite subgraph $F$ of $G$ induced by the vertices $A\cup
B $. Assign to each vertex $u$ in $F$ the list $L(u)=\{156,312,468\}$, and
apply Theorem \ref{ListBip} with function $q(u)=S_{h}(u)$. Let $f$ be a
coloring satisfying the assertion of Theorem \ref{ListBip}. That is, $f$
satisfies condition $S_{f}(u)+S_{h}(u)\neq S_{f}(v)+S_{h}(v)$ for every edge 
$uv\in E(F)$, where%
\begin{equation*}
S_{f}(u)=\dsum\limits_{x\in N_{F}(u)}f(x).
\end{equation*}%
Putting things together we define a function $g$ on the whole set of
vertices $V(G)$ by joining $f$ and $h$:%
\begin{equation*}
g(x)=\left\{ 
\begin{array}{l}
h(x)\text{ if }x\in C\cup D \\ 
f(x)\text{ if }x\in A\cup B%
\end{array}%
\right. .
\end{equation*}%
We claim that $g$ is an additive coloring of $G$ over the set $\{1,2,\ldots
,468\}$. Let $S(u)$ be the sum of $g$-labels over the whole neighborhood $%
N(u)$, that is, $S(u)=S_{h}(u)+S_{f}(u)$. Let $uv$ be any edge in $G$. If $%
u\in A\cup B$ and $v\in C\cup D$, then $S_{h}(u)\neq 0(\func{mod}156)$ while 
$S_{f}(u)=0(\func{mod}156)$, thus $S(u)\neq 0(\func{mod}156)$. The other end
of the edge satisfies $S_{h}(v)=S_{f}(v)=0(\func{mod}156)$, so $S(v)=0(\func{%
mod}156)$. If $u\in A$ and $v\in B$, condition $S(u)\neq S(v)$ is guaranteed
by construction of $f$. We are left with the last case $u\in C$ and $v\in D$%
. Suppose on the contrary that $S(u)=S(v)$. Since $S_{f}(u)=S_{f}(v)=0(\func{%
mod}156)$, we get $S_{h}(u)=S_{h}(v)$ in $\mathbb{Z}_{156}$. But $%
S_{h}(u)=\sigma ((0,S_{h_{2}}(u)))$ and $S_{h_{2}}(u)\neq 0$ in $\mathbb{Z}%
_{13}$, while $S_{h}(v)=\sigma ((S_{h_{1}}(v),0))$. This contradiction
completes the proof.
\end{proof}

A set of vertices $I$ in a graph $G$ is called \emph{two-independent} if the
distance between any two vertices of $I$ is at least three. In \cite%
{BuCranston} it was proved that every planar graph of girth at least $13$
has a vertex decomposition into two sets $I$ and $F$ such that $I$ is
two-independent and $F$ induces a forest. Our last theorem follows easily
from this result.

\begin{theorem}
\label{Girth}Every planar graph of girth at least $13$ satisfies $\eta
(G)\leqslant 4$.
\end{theorem}

\begin{proof}
Let $V(G)=I\cup F$, where $I$ is $2$-independent and $F$ induces a forest.
By Corollary \ref{TreeBP} there is an additive coloring $f$ of the forest $F$
using labels $\{2,4\}$. Extend this coloring to the whole graph $G$ by
putting $f(i)=1$ for each vertex $i\in I$. It is easy to see that $f$ is an
additive coloring of $G$.
\end{proof}

\section{Finite abelian groups}

The problem of additive coloring can be considered in a more general setting
of Abelian (additive) groups. We may use elements of any such group $\Gamma $
as the labels of vertices and define the additive coloring the same way as
before. Accordingly to our main conjecture, as well as to the methods we
develop so far, one could expect that perhaps every graph has an additive
coloring over some group whose order is equal to the chromatic number of the
graph. We prove below that this is not true.

\begin{theorem}
For every $r\geqslant 2$ there is a graph $G_{r}$ such that $\chi (G_{r})=r$%
, and there is no additive coloring of $G_{r}$ over any finite Abelian group
of order $r$. But there is an additive coloring of $G_{r}$ in $\mathbb{Z}%
_{r+1}$.
\end{theorem}

\begin{proof}
Let $P$ denote a path on five vertices $a,x,b,y,c$ (in that order). Consider
a graph $H=H(r)$ obtained by blowing up each of the two vertices $x$ and $y$
to the clique $K_{r-1}$. Now, take $r$ copies of $H$, chose one vertex $%
v_{i} $ in any of the two cliques $K_{r-1}$ in each copy of $H$, and join
all these vertices mutually to form a new clique $K_{r}$. We claim that in
this way we constructed a graph $G_{r}$ satisfying the assertion of the
theorem. It is not hard to see that $\chi (G_{r})=r$. To prove the first
part of the theorem, suppose that $\Gamma $ is any Abelian group of order $r$%
, and there is a coloring $f:V(G_{r})\rightarrow \Gamma $ such that the sums 
$S(v)$ form a proper coloring of $G_{r}$. Notice that in any proper coloring
of $H$ with $r$ colors, the vertices $a$, $b$, and $c$ must have the same
color. Thus $s(a)=s(b)=s(c)$. Notice also that, by the definition of
additive coloring we have $S(b)=S(a)+S(c)$, which implies that $%
S(a)=S(b)=S(c)=0$ in every copy of $H$ in $G_{r}$. This implies in turn that 
$S(v)\neq 0$ for all other vertices of $G_{r}$. In particular, we get a
proper coloring of the clique $K_{r}$ by non-zero elements of $\mathbb{Z}%
_{r} $, which is not possible.

For the second assertion we define explicitly an additive coloring $%
f:V(G_{r})\rightarrow \mathbb{Z}_{r+1}$ as follows. Denote by $H_{i}$ the $i$%
th copy of the graph $H$ in $G_{r}$. Let $X_{i}$ and $Y_{i}$ denote the two
cliques $K_{r-1}$ in $H_{i}$ obtained by blowing up the vertices $x$ and $y$%
, respectively. Also, let $a_{i}$, $b_{i}$, and $c_{i}$ be the respective
copies of the end vertices and the middle vertex of the path $P$ in $H_{i}$.
Finally, let $v_{i}$ denote the unique vertex of $H_{i}$ belonging to the
clique $K_{r}$. We may assume that $v_{i}\in V(X_{i})$. We have to
distinguish two cases.

\begin{enumerate}
\item \emph{(The number }$r+1$\emph{\ is odd.) }Put $f(v_{i})=f(b_{i})=0$
and $f(a_{i})=f(c_{i})=i$ for all $i=1,2,\ldots ,r$. Then extend injectively
the coloring using all labels from the set $\{1,2,\ldots ,r\}\setminus
\{i,-i\}$ on each of the two cliques $X_{i}$ and $Y_{i}$. So, the total sum
of labels in each of the cliques $X_{i}$ and $Y_{i}$ is zero. Hence, we get $%
S(v_{i})=i$ and $S(a_{i})=S(b_{i})=S(c_{i})=0$. For any other vertex $u$ we
get $S(u)\neq 0$. Also, we cannot have conflicts inside cliques $X_{i}$ and $%
Y_{i}$ by injectivity.

\item \emph{(The number }$r+1$\emph{\ is even.)} Let $r+1=2k$. First we
construct our coloring on all copies $H_{i}$ for $i\neq k$. Put $%
f(v_{i})=f(b_{i})=f(c_{i})=0$ and $f(a_{i})=i$. Extend injectively the
coloring on the clique $X_{i}$ using all labels from the set $\{1,2,\ldots
,r\}\setminus \{i,-i\}$. So, the total sum of labels on $X_{i}$ is equal to $%
k$. Next, extend the coloring injectively to cliques $Y_{i}$ using all
labels from the set $\{1,2,\ldots ,r\}\setminus \{k\}$. Hence, the total sum
of labels over $Y_{i}$ is zero. Thus we get $S(v_{i})=k+i$, $%
S(a_{i})=S(b_{i})=k$, and $S(c_{i})=0$ for all $i\neq k$. For $u\in X_{i}$
we have $S(u)=k+i-f(u)\neq k$, since $f(u)\neq i$. For $u\in Y_{i}$ we have $%
S(u)=-f(u)\neq 0,k$. Also there are no conflicts inside cliques $X_{i}$ and $%
Y_{i}$ by injectivity. It remains to extend the coloring to the copy $H_{k}$%
. Put $f(v_{k})=0$, $f(a_{k})=1$, $f(b_{k})=k$, and $f(c_{k})=k-1$. Next put
injectively all labels from the set $\{1,2,\ldots ,r\}\setminus \{k,k+1\}$
to the vertices of $X_{k}$, and similarly for $Y_{k}$ using the set $%
\{0,1,\ldots ,r\}\setminus \{k,r\}$. So, the total sum over $X_{k}$ is $k-1$
and the total sum over $Y_{k}$ is $1$. Hence, we get $S(a_{k})=k-1$, $%
S(c_{k})=1$, $S(b_{k})=k$, and $S(v_{k})=0$. Since each vertex $u\in
X_{k}\cup Y_{k}$ satisfies $S(u)=-f(u)$, no other conflicts could appear.
\end{enumerate}

The proof is complete.
\end{proof}

Notice that graph $G_{4}$ from the above proof is planar, so we cannot get
our main conjecture for planar graphs using finite groups. Notice also, that 
$G_{2}$ is a tree, and $G_{3}$ is an outer planar graph, so the same
difficulty is true for planar graphs with smaller chromatic number. Perhaps
every $r$-colorable graph has an additive coloring modulo $r+1$.

We conclude this section with the following simple result.

\begin{theorem}
Let $A$ be a fixed Abelian group. The problem of deciding whether a given
graph $G$ has an additive coloring over $A$ is NP-complete if $\left\vert
A\right\vert \geqslant 3$, and polynomial for $A=\mathbb{Z}_{2}$.
\end{theorem}

\begin{proof}
Let $\left\vert A\right\vert =k\geqslant 3$. For a given graph $G$, whose
vertex set is $V(G)=\{v_{1},\ldots ,v_{n}\}$, consider a new graph $%
G^{\prime }$ obtained by adding $n$ new vertices $\{v_{1}^{\prime },\ldots
,v_{n}^{\prime }\}$ and $n$ new edges $v_{i}v_{i}^{\prime }$ for $i=1,\ldots
,n$. We prove that $G$ is $k$-colorable (in the usual sense) if and only if $%
G^{\prime }$ is additively colorable over $A$. This will prove the first
assertion of the theorem.

Obviously, if $G^{\prime }$ has an additive coloring over $A$, then $G$ is $k
$-colorable in the usual sense. For the other implication, assume that $G$
is $k$-colorable, and fix a proper coloring $c$ of $G$ using $A$ as the set
of colors. Now fix a nonzero element $a\in A$ and define a new coloring $f$
of $G^{\prime }$ in the following way:

\begin{enumerate}
\item If $c(v_{i})=0$, then $f(v_{i})=a$.

\item If $c(v_{i})\neq 0$, then $f(v_{i})=0$.

\item $f(v_{i}^{\prime })=c(v_{i})-\dsum\limits_{x\in N_{G}(v_{i})}f(v_{i})$.
\end{enumerate}

We claim that $f$ is a desired additive coloring of $G^{\prime }$ over $A$.
Indeed, the sum of colors around each vertex $v_{i}$ satisfies%
\begin{equation*}
S(v_{i})=\dsum\limits_{x\in N_{G}(v_{i})}f(v_{i})+f(v_{i}^{\prime
})=c(v_{i}),
\end{equation*}%
so there are no conflicts in $G$. Also by definition of $f$ we have 
\begin{equation*}
S(v_{i}^{\prime })=f(v_{i})\neq c(v_{i})=S\left( v_{i}\right) 
\end{equation*}%
for each vertex $v_{i}^{\prime }$. This prove the claim.

For the second assertion just notice that the problem reduces to recognizing
if a given graph $G$ is bipartite, and then checking solvability of a system
of linear equations of the form $Mx=y$ over $\mathbb{Z}_{2}$, where $M$ is
the adjacency matrix of $G$, and $y$ is binary vector encoding a proper
coloring of $G$. There are actually two possible such vectors for a
connected bipartite graph $G$. This completes the proof.
\end{proof}

\section{Open problems}

We conclude the paper with a short list of open questions concerning
additive coloring of graphs.

\begin{conjecture}
Every graph $G$ satisfies $\eta (G)\leqslant \chi (G)$.
\end{conjecture}

It is not known whether this is true for bipartite graphs. It is not even
known if $\eta (G)$ is bounded for bipartite graphs. A heuristic argument is
that the statement of the conjecture holds trivially if we extend the set of
labels to real numbers. Indeed, any proper coloring of a $k$-colorable graph 
$G$ with a set of $k$ real numbers which is independent over rationals,
gives an additive coloring of $G$. Another direction is to consider additive
colorings in finite Abelian groups.

\begin{conjecture}
Every graph $G$ has an additive coloring modulo $\chi (G)+1$.
\end{conjecture}

If true this is best possible, as we proved in section 5.

Our last problem arose as a vertex analog of the famous \emph{antimagic
labeling conjecture} of Ringel \cite{HartsfieldRingel}.

\begin{conjecture}
Let $G$ be a simple graph on $n$ vertices in which no two vertices have the
same neighborhood. Then there is a bijection $f:V(G)\rightarrow \{1,2,\ldots
,n\}$ such that%
\begin{equation*}
\dsum\limits_{x\in N(u)}f(x)\neq \dsum\limits_{x\in N(v)}f(x)
\end{equation*}%
for any two distinct vertices $u$ and $v$.
\end{conjecture}

\begin{acknowledgement}
Sebastian Czerwi\'{n}ski and Jaros\l aw Grytczuk acknowledge a partial
support from Polish Ministry of Science and Higher Education Grants (MNiSW)
(N N201 271335) and (MNiSW) (N N206 257035). 
\end{acknowledgement}

\end{document}